\documentclass[11pt,english]{article}
\usepackage{hyperref}
\usepackage{amsthm,amsmath,graphicx,amssymb,amscd,lmodern}
\usepackage[utf8]{inputenc}
\usepackage{calrsfs}

\newtheorem{thm}{Theorem}[section]
\newtheorem{PROP}{Proposition}[section]

\newtheorem{prop}[thm]{Proposition}
\newtheorem{COR}[PROP]{Corollary}

\newtheorem{cor}[thm]{Corollary}
\newtheorem{rem}[thm]{Remark}
\newtheorem{lem}[thm]{Lemma}

\newcommand{\lie}[1]{\mathfrak{#1}}
\newcommand{\Gtwo}{G_2}
\newcommand{\R}{{\mathbb R}}

\newcommand{\Spin}{\mathrm{Spin}}
\newcommand{\GL}{\mathrm{GL}}

\newcommand{\SO}{{\rm SO}}
\newcommand{\so}{\lie{so}}
\newcommand{\Sym}{\rm Sym}
\newcommand{\vg}{\mathrm{vol}^g}
\newcommand{\Cal}[1]{\mathcal{#1}}

\usepackage[affil-it]{authblk}

\title{An energy functional on the universal spinor bundle} 

\author{Leonardo Bagaglini}
\affil{
Dipartimento di Matematica e Informatica \lq \lq Ulisse Dini\rq \rq, Universit\`a degli Studi di Firenze, Viale Giovan Battista Morgagni, 67/A , 50134 Firenze, Italy\\\par
E-mail address: \emph{leonardo.bagaglini@unifi.it}}

\begin{document} 
 \maketitle
\begin{abstract}  
We study an energy functional on the universal spinor bundle over a closed $n$-dimensional spin manifold $M$. The critical points of this functional, which is modelled on the total torsion functional of $\Gtwo$-structures in seven dimensions, are pairs of Ricci-flat metrics and real parallel spinor fields provided that $n$ equals $3$ or $7$. We then modify the functional to obtain the analogue in arbitrary dimensions. Finally we apply the universal spinor bundle approach to solve some ODEs problems concerning $\Gtwo$-structures.      
\end{abstract} 

\tableofcontents


\section*{Introduction}
In \cite{BG92} Bourguignon and Gauduchon introduced the universal spinor bundle related to a spin manifold in order to understand spinor fields and Dirac operators affected by metric variations. Given a $n$-dimensional spin manifold and a spin module $\Delta_n$ the universal spinor bundle $\Sigma$ is a fiber bundle over $M$ with fiber modelled on $S^2_+(\R^n)^*\times\Delta_n$, so its sections can be interpreted as pairs of Riemannian metrics and spinor fields. \par
In \cite{Am16} Ammann, Weiss and  Witt studied the natural energy functional, defined on the space of its sections, given by
$$E(\psi)=\frac{1}{2}\int_M|\nabla^g\psi|^2_g \vg.$$
In the closed case they proved that the set of its critical points consists of absolute minimizers only: namely parallel spinors of Ricci-flat metrics. Consequently the problem of finding critical points of $E$ is strictly related to the special holonomy problem of Levi-Civita connections.\par
In particular the case $n=7$ concerns the $\Gtwo$-holonomy: metric connectionsw with restricted Holonomy group contained in the seven-dimensional irreducible representation of $\Gtwo$, the smallest compact exceptional simple Lie group.\par
There are several ways to approach the $\Gtwo$-geometry: through differential forms ( \cite{Bry}, \cite{Hit03}, \cite{Gri} ) or spinor fields (\cite{Fri97}, \cite{Chi},\cite{Ba17}). Variational problems may be hard to treat with the spin formalism, since deformations of $\Gtwo$-structures (in the sense of principal bundles) generically produce deformations of the underlying Riemannian metrics. However, the Bourguignon and Gauduchon's approach consents to deal with this issue.\par
In analogy with the full torsion functional of $\Gtwo$-structures (see \cite{WW12} or \cite{Ba17}) we introduce the following:
$$\Cal{E}(\psi)=\frac{1}{2}\int_M |T^\psi|^2_g\vg,$$
where $T^\psi$ is the $(2,0)$-tensor depending on the (real) unit $g$-spinor field $\psi$ by
$$T^\psi(X,Y)=<\nabla^g_X\cdot\psi,Y\cdot\psi>,\quad X,Y\in\Gamma(M,TM).$$
Then we prove the following proposition.
\begin{PROP}
For $n=3$ or $7$ the critical points of $\Cal{E}$ are $g$-parallel unit spinor fields.
\end{PROP}
\par
This functional, for $n=7$, turns out to coincide with that considered in \cite{WW12} (up to different torsion conventions) through the equivalence between positive forms (\cite{Hit03}) and (projective) spinor fields. In the spinorial setting it was partially studied by the author in \cite{Ba17}, where he considered variations inside a fixed spinor bundle.\par 
Here, for $n\geq 3$, we introduce two tensors, denoted by $T_n$ and $T_{n-1}$, that are able to encode all the information about $\nabla^g\psi$. Precisely the absolute minimizers of the functional $\Cal{E}_n+\Cal{E}_{n-1}$ related to them are always $g$-parallel unit spinors. These tensors are essentially modelled on $T^\psi$ and are constructed to catch all the components of $\nabla^g\psi$ inside $\psi^\perp$. We obtain the following result.
\begin{PROP}
The critical points of the energy functional $\Cal{E}_n+\Cal{E}_{n-1}$ are $g$-parallel unit spinor fields.
\end{PROP}
Furthermore we study the (negative) gradient $\Cal{Q}$ of $\Cal{E}_{n}+\Cal{E}_{n-1}$ with respect to the natural $L^2$-structure on $\Gamma(M,\Sigma)$. As in \cite{Am16} we find out that the principal symbol of $\Cal{Q}$ is positive definite on the orthogonal complement of its kernel, which coincides with the vector space generated by the tangent action of $\widetilde{\mathrm{Diff}}^0(M)$ on $\Sigma$. Moreover the Hessian of $\Cal{E}_n+\Cal{E}_{n-1}$ at any critical point is positive semi-definite.\par
These results are indeed expected due to the equivariance of the functional under spin diffeomorphisms. Consequently, using the Hamilton-De Turk technique (\cite{Ham1}, \cite{Ham2}, \cite{DeT83}), we succeed to prove the short-time existence of the gradient flow.
\begin{PROP}
There exists a positive real number $\epsilon $ such that the gradient flow 
$$\begin{cases}
\dot{\psi}_t=\Cal{Q}(\psi_t),\\
\psi_0=\psi,
\end{cases}$$
has unique solution in $[0,\epsilon)$.
\end{PROP}\par
Finally, in the last section, we restricts our attention to the space of $\Gtwo$-structures $\Omega^3_+\subset\Omega^3(M)$ ($n=7$). Precisely we study the relation between the horizontal distribution of Bourguignon and Gauduchon and the invariant distributions on the space $\Omega^3_+$. As expected the horizontal distribution coincides with $\omega\mapsto\Omega^3_{28}(\omega)$ in $\Omega^3_+$ under the equivalence between fundamental three-forms and fundamental spinor fields of $\Gtwo$-structures.
In particular this last distribution is not integrable.
We then deduce the followings.
\begin{COR}
Let $g_t$ be a smooth one-parameter family of Riemannian metrics and let $\omega$ be the fundamental three-form of a $\Gtwo$-structure with underlying metric $g_0$. Then the solution of the ODE
\begin{equation*}
\begin{cases}
\dot{\omega}_t=\frac{3}{2}i_{\omega_t}(\dot{g}_t),\\
\omega_0=\omega,
\end{cases}
\end{equation*}
exists as long as $g_t$ does and belongs to $\Omega^3_+$.
Moreover, if $(g_t)_x$ lies in a maximal flat subspace of $S^2_+T^*_xM$ for any $x\in M$ and $t\in[0,t_1]$, then $\omega_{t_1}$ depends only on $g_0,g_{t_1}$ and the flat subspace. 
\end{COR}
\begin{COR}
Let $\omega$ be the fundamental three-form of a $\Gtwo$-structure with underlying metric $g$ and let $h$ be a nowhere vanishing $(2,0)$-tensor. We denote by $\Cal{R}_g^h$ the space of Riemannian metrics $k$ such that the $g$-symmetric endomorphisms $k_{ai}g^{ib}$ and $h_{ai}g^{ib}$ commute. 
Then there exists a embedding, in the sense of Fréchet,
$$\begin{CD}
\Cal{B}^\omega:\;\Cal{R}_g^h@>>>\Omega^3_+,
\end{CD}$$
sending a metric $k$ into a compatible positive three-form, satisfying $\Cal{B}^\omega(g)=\omega$ and having tangent map
$$D\Cal{B}_k(k')=3/2\,i_{\Cal{B}(k)}(k'),\quad k'\in T_k\Cal{R}^g_h.$$ 
\end{COR}
\par\bigskip
The paper is structured as follows. In Section \S\ref{sec1} we review basic definitions, notations and properties about tensor spaces and Spin Geometry.\par 
Later, in Section \S\ref{sec2} we introduce the first energy functional. There we study its gradient and we discuss the nature of its critical points.\par
In Section \S\ref{sec3}, in analogy with the previous one, we define a new energy functional. We analyse its first variation and prove short-time existence and uniqueness of the relative gradient flow.\par
Finally, in Section \S\ref{sec4}, we discuss the relation between the natural connection of Bourguignon and Gauduchon and the invariant distributions on the space of $\Gtwo$-structures.  
\section{Algebraic setting and Spin Geometry}\label{sec1}
Let $g$ be the euclidean metric on $\R^n$, equipped with its metric volume form $\vg=\star_g 1$, and let $\SO(n)$ denote the special orthogonal group. Whenever a basis $\left\{e_1,\dotsc,e_n\right\}$ of $\R^n$ is chosen we denote by $\left\{e^1,\dotsc,e^n\right\}$ its dual basis. With respect to this choice we write
$$x=x^ae_a=\sum_{a=1}^nx^ae_a,\quad x\in\R^n,$$
where we are adopting the Einstein convention. We henceforth identify vectors and one-forms in the tensor spaces $T^{(p,q)}=\otimes^p(\R^n)^*\bigotimes\otimes^q(\R^n)$ via the Riesz isomorphism induced by $g$. For instance a $(2,0)$-tensor $T_{ab}$ can be identified with a $(1,1)$-tensor as follows
$$T_{a}^{\phantom{a}b}=T_{ai}g^{ib},$$
where $g^{ib}$ denotes the inverse matrix of $g_{ib}$.\par 
The space of alternating $r$-forms, or simply $r$-forms, $\Lambda^r(\R^n)^*$, is generated by simple elements 
$$e^{a_1}\wedge\dotsc\wedge e^{a_r}=\frac{1}{r!}\sum_{\sigma\in S_r}\mathrm{sgn}(\sigma_r)e^{a_{\sigma_1}}\otimes\dotsc\otimes e^{a_{\sigma_r}}.$$\par 
Similarly the space of symmetric two-forms $S^2(R^n)^*$ is generated by simple elements
$$e^1\odot e^2=\frac{1}{2}\left(e^1\otimes e^2+e^2\otimes e^1\right).$$
We further denote the open subset of positive definite symmetric two-forms by $S^2_+(\R^n)^*$.\par
We will sometimes symmetrise, or skew-symmetrise, a tensor $T$ of type $(r,0)$. So $T_{(a_1\dots a_r)}$ and $T_{[a_1\dots a_r]}$ will denote the total symmetrisation and skew-symmmetrisation of $T$ respectively; in coordinates:
$$T_{(a_1\dots a_r)}=\frac{1}{r!}\sum_{\sigma\in S_r}T_{a_{\sigma_1}\dotsc a_{\sigma_r}},\quad
T_{[a_1\dots a_r]}=\frac{1}{r!}\sum_{\sigma\in S_r}\mathrm{sgn}(\sigma)T_{a_{\sigma_1}\dotsc a_{\sigma_r}}.$$ 
We also define the contraction of a vector $x$ with a $(2,0)$-tensor $T$ to be the vector defined by
$$(v\neg T)_{a}=v^iT_{ia}.$$
Finally, for our convenience, given a tensor $T$ of type $(r+1,0)$ we set
$$(T*_{g}T)_{ab}^{\phantom{j}}=T^{\phantom{a}}_{ai_1\dotsc i_r}T_{b}^{\phantom{b}i_1\dotsc i_r};$$
which defines a tensor of type $(2,0)$. \par
Note that al these vector spaces inherits a metric from $\R^n$ and which, by abuse of notation, we will continue to denote by $g$. 
\par\bigskip
We now give a brief summary of the Spin Geometry, referring to \cite{SG} for all the contents to follow.\par  
Let $\GL^+$ be the identity component of the general linear group of $\R^n$, for $n\geq 3$, and let $\pi:\widetilde{\GL}^+\rightarrow\GL$ its universal covering. As usual we choose $\SO(n)$ and its lifting $\Spin(n)$ as maximal compact subgroups of $\GL^+$ and $\widetilde{\GL}^+$ respectively.  The polar decomposition directly implies that the Mostov fibrations relatives to $\widetilde{\GL}^+$ and $\GL^+$ are trivial; furthermore it shows that these two spaces are equivalent\footnote{As $\SO(n)$ and $\Spin(n)$-spaces. The actions are the adjoint ones.} to ${\Sym}(n)\times \SO(n)$ and ${\Sym}(n)\times\Spin(n)$, where ${\Sym}(n)$ denotes the space of symmetric positive definite endomorphisms on $\R^n$ and which equivariantly identifies with $S^2_+(\R^n)^*$. In particular both $\widetilde{\GL}^+$ and $\GL^+$ are principal bundles, with structure group $\Spin(n)$ and $\SO(n)$ respectively, over $S^2_+(\R^n)^*$. Moreover any tangent space to $\GL^+$, or $\widetilde{\GL}^+$, turns to be equivalent to $S^2(\R^n)^*\oplus\so(n) $, being $\so(n)$ the Lie algebra of $\SO(n)$. Since each $\so(n)$ represents the tangent space to the fiber in the previous fibration, all of them form the vertical distribution. Therefore the union of all the spaces isomorphic to $S^2(\R^n)^*$ defines a horizontal distribution $H$, actually a connection on each bundle. Following \cite{BG92} we say $H$ the \emph{natural connection} on $\GL^+$, or $\widetilde{\GL}^+$.\par
Finally let $\Delta_n$ denote a real spin module, whose elements are called spinors. This is a real irreducible module for the action of the real Clifford algebra ${\rm Cl}_n$, in which $\Spin(n)$ naturally embeds. Recall that $\mathrm{ Cl}_n$, as vector space, can be identified with the Grassmann algebra $\Lambda^\bullet(\R^n)^*$; in particular there is an action of vectors, or one-forms, on $\Delta_n$ called the \emph{Clifford multiplication}. This action, denoted by $\cdot$, satisfies the following identity
$$x\wedge y\cdot v= x\cdot y \cdot v+g(x,y)v,\quad x,y\in\R^n,\;v\in\Delta_n,$$
and turns out to be $\Spin(n)$-invariant.\par
Regarded as a $\Spin(n)$-module $\Delta_n$ is irreducible for $n=3,6,7$ mod(8) whereas, for the other values of $n$, there is a decomposition into two irreducible modules with the same dimension; they are equivalent if $n=1,2$ mod(8) and inequivalent if not. In any case $\Delta_n$ inherits a unique (up to a positive constant on each irreducible module) $\Spin(n)$-invariant scalar product $<.,.>$, which is compatible with the Clifford multiplication in the following sense
$$<x\cdot v,w>=-<v,x\cdot w>,\quad x\in\R^n,\;v,w\in\Delta_n.$$\par
We now pass to the global considerations. Let $M$ be a closed (compact with no boundary), connected, seven-dimensional spin manifold. We choose a spin structure $\pi:\widetilde{\Cal{P}}\rightarrow\Cal{P}$ on it; in other words we require the existence of the following data: a principal\footnote{We consider right actions on the bundles and left actions on the modules.} $\widetilde{\GL}^+$-bundle $\widetilde{\Cal{P}}$ over $M$ together with a two-fold covering map over the $\GL^+$-principal bundle of oriented frames $\Cal{P}$ compatible with $\pi:\widetilde{\GL}^+\rightarrow\GL$. As in \cite{BG92} we define the \emph{universal spinor bundle} $\Sigma$ of $M$ to be 
$$\Sigma=\widetilde{\Cal{P}}\times_{\Spin(n)}\Delta_n=\left\{\left.\left[{{\sigma}},v\right]\,\right|\,{{\sigma}}\in \widetilde{\Cal{P}},\,v\in\Delta_n\right\}.$$ 
This manifold has a natural structure of vector bundle, with fiber modelled on $\Delta_n$, over $S^2_+T^*M$, the fiber bundle of positive definite symmetric two-forms on $M$, due to the identification of the last with $\widetilde{\Cal{P}}/\Spin(n)$. Therefore we can define the \emph{natural horizontal distribution} to be the distribution generated by the horizontal spaces of the natural connection; this definition makes sense in virtue of the equivariance laws. Note that this distribution \emph{does not} define a principal connection on $\widetilde{\Cal{P}}\rightarrow S^2_+T^*M$, but a merely partial connection. Indeed $T\widetilde{\Cal{P}}$ locally identifies with  $\so(n)\times S^2(\R^n)^*\times \R^n$, where $\so(n)$ is the vertical space and $S^2(\R^n)^*$ the horizontal one just defined; so, taking $x\in M$, there is no a prescribed way to lift a tangent vector of $S^2_+T^*M$, which is transversal to the fiber over $x$, to $T\widetilde{\Cal{P}}$.\par 
However the horizontal distribution descends to a distribution, again called \emph{natural horizontal} and denoted by $H$, on the vector bundle
$$
\begin{CD}
\Sigma@>\Theta>>S^2_+T^*M.
\end{CD}
$$
The vertical distribution of this vector bundle will be denoted by $V$.\par
Let us observe that any Riemannian metric $g$ on $M$ gives rise to a reduction $\widetilde{\Cal{P}}^g$ of $\widetilde{\Cal{P}}$ with structure group $\Spin(n)$, thanks to which we are able to define the associated spinor bundle over $M$
$$
\Sigma^g=\widetilde{\Cal{P}}^g\times_{\Spin(n)}\Delta_n\subset\Sigma.
$$
We consider each $\Sigma^g$ equipped with the real metric $<.,.>_g$, of related pointwise norm $|.|_g$, inherited by $\Delta_n$. Moreover the Levi-Civita connection $\nabla^g$ on $\Cal{P}^g$ lifts to a metric connection on $\widetilde{\Cal{P}}^g$ which allows us to define a connection on $\Sigma^g$. This connection, which we continue to denote $\nabla^g$, turns out to be compatible with $<.,.>_g$ and the Clifford multiplication.\par
The composition of $\Theta:\Sigma\rightarrow S^2_+T^*M$ with the natural projection of $S^2_+T^*M$ onto $M$ gives a map
$$
\begin{CD}
\Sigma@>\theta>> M,
\end{CD}
$$
defining a structure of fiber bundle over $M$ on $\Sigma$, whose fibers are modelled on $S^2_+(\R^n)^*\times\Delta_n$. A (local) section ${s}$ of this bundle can be therefore thought as a couple $(g,\psi)$ where, for each $x\in M$, $g_x$ is a metric in $T_xM$ and $\psi_x$ is a spinor in $\Sigma^g_x$. Indeed ${s}=[{{\sigma}},v]$ for ${{\sigma}}\in\widetilde{\Cal{P}}$ and $v\in\Delta_n$, so $g=\Theta\left({s}\right)$ and $\psi=s$ when considered as a section of $\Sigma^g$. Conversely a metric $g$ gives rise to a spinor bundle $\Sigma^g$ which embeds in $\Sigma$; then a $\Sigma^g-$valued spinor field $\psi$ defines the desired section ${s}$. Properly speaking $\psi$ and $s$ coincide as sections over $M$, however we use two different symbols to emphasize the dependence on the metric of the first.\par
Thanks to the vector bundle structure we can identify the vertical distribution with the distribution generated by
$$V_{s}=\Sigma_g,\quad {s}\in\Sigma,\;g=\Theta(s).$$
On the other hand the horizontal distribution can be explicitly described as follows (\cite{BG92}). Let ${s}=[{\sigma},v]\in\Sigma$ and let $g_t$ be a curve of metrics with $\Theta({s})=g_0$ and $\dot{g}_0=h$. With no loss of generality we can choose $g_t=g_0+th$. The endomorphism $A_{g_t}^{g_0}$ defined by
$$g_0(X,Y)=g_t(A_{g_t}^{g_0}X,Y),\quad X,Y\in\Gamma(M,TM),$$ 
is $g_t$-symmetric and positive-definite; in particular it has a unique $g_t$-symmetric positive-definite square root $B_{g_t}^{g_0}\in\mathrm{GL}(TM)$ (whose action turns $g_0$ into $g_t$); in coordinates $B_{g_t}^{g_0}$ is simply given by the positive root of $(g_0)^{\phantom{m}}_{am}(g_t)_{\phantom{m}}^{mb}$. This endomorphism allows us to define a curve in $\widetilde{\Cal{P}}$ by lifting $B_{g_t}^{g_0}\pi({\sigma})$ starting from ${\sigma}$; let us denote it by $\widetilde{B}_{g_t}^{g_0}{\sigma}$. Then it can be shown that the horizontal lifting of $g_t$ starting at ${s}\in\Sigma$ is $[\widetilde{B}_{g_t}^{g_0}{\sigma},v]$. Moreover if we replace $g_t=g+th$ with another curve $g_t'$ having the same end points and such that $(g_t')_x$ lies in the maximal flat subspace containing $(g_t)_x$ inside the non-positive curved symmetric space $S^2_+T^*_xM$, then the relative horizontal displacements coincide. Therefore we have the following characterization
$$H_{s}=\left\{\left.\left.\frac{d}{dt}\right|_{t=0}\left[\widetilde{B}_{g_t}^{g_0}{\sigma},v\right]\;\right|\;g_t=g_0+th,\;h\in S^2T^*M\right\},\quad {s}\in \Sigma.$$
Note also that the map $\widetilde{B}_{g_t}^g$ is equivariant with respect to the Clifford multiplication, that is 
$$\widetilde{B}^g_{g_t}\left(\omega\cdot s\right)=B^g_{g_t}\omega\cdot\widetilde{B}^g_{g_t}s,\quad \omega\in\Lambda^\bullet T^*M,\;s\in \Sigma.$$
\par
The space of smooth sections $\Cal{S}=\Gamma(M,\Sigma)$ is a Fréchet bundle over $\Cal{R}=\Gamma(M,S^2_+T^*M)$, with vertical distribution $\Cal{V}$, and carries a connection, inherited by the natural one, with horizontal distribution $\Cal{H}$; moreover the subset $\Cal{S}_1=\Gamma(M,\Sigma_1)$, where $\Sigma_1=\left\{(g,\psi)\,|\,|\psi|_g\equiv 1\right\}$, turns to be a Fréchet subbundle of the former. It is clear that the vertical and horizontal distributions are identified with
$$\Cal{V}_{s}=T_\psi\Gamma(M,\Sigma^g)=\Gamma(M,\Sigma^g),\quad \Cal{H}_{s}=T_g\Cal{R}=\Gamma(M,S^2M),\quad {s}\in\Cal{S}.$$\par
We finally consider $\Cal{S}$ equipped with the following metric structure: for any ${s}=(g,\psi)\in \Cal{S}$, $h_1,h_2\in\Cal{H}_{s}$ and $\psi_1,\psi_2\in \Cal{V}_{s}$ we define
$$\left\langle\left\langle (h_1,\varphi_1),( h_2,\varphi_1)\right\rangle\right\rangle_{s}=\int_M g(h_1,h_2)+<\varphi_1,\varphi_2>_g \mathrm{vol}^g,$$
where $\vg$ is the metric volume form associated to $g$ and the spin structure.
\section{First energy functional}\label{sec2}
Recall the notation of the previous section and let $M$ be a closed, connected, $n$-dimensional spin manifold equipped with a spin structure. Further let $g\in\Cal{R}$ and $\psi\in\Gamma(M,\Sigma^g)$. We define the \emph{energy density tensor} $T^\psi$ of $\psi$ by
\begin{equation}\label{T}
T^\psi(X,Y)=<\nabla^g_X\psi,Y\cdot \psi>_g,\quad X,Y\in\Gamma(M,TM),
\end{equation}
the related energy functional $\mathcal{E}$ on $\Cal{S}_1$ by
\begin{equation}\label{energy}
\Cal{E}({s})=\frac{1}{2}\int_M \left|T^\psi\right|_g^2\vg,\quad {s}=(g,\psi)\in\Cal{S}_1,
\end{equation}
and the associated gradient $Q$ by
$$\left\langle\left\langle Q({s}),\dot{{s}}\right\rangle\right\rangle_{s}=-D\Cal{E}_{s}(\dot{{s}}).$$\par
Before studying the functional let us give some definitions, which will be useful later on. We henceforth identify the $(2,0)$-tensor $(T^\psi)_{ab}$ with the $(1,1)$-tensor $(T^\psi)_a^{\phantom{a}b}$ by using $g$. We denote by $F^\psi$ the $(0,3)$-tensor given by
\begin{equation}\label{F}
4(F^\psi)(X,Y,Z)=<X\wedge Y \cdot \psi, T^\psi(Z)\cdot\psi>_g,\quad X,Y,Z\in\Gamma(M,TM).
\end{equation}
Finally we set $D^\psi$ to be the spinor field
\begin{equation}\label{D}
D^\psi=2(T^\psi)^{ab}e_b\cdot \nabla^g_a\psi,
\end{equation}
where $\left\{e_1,\dots,e_n\right\}$ is any orthonormal frame. 
\begin{rem}
When $n=7$ the vector space $\Delta_7$ turns out to be the (unique) irreducible real spin module and, whenever $\psi\neq 0$ has constant length, the energy tensor $T^\psi$ represents the full torsion tensor of the $\Gtwo$-structure associated to $\psi$ (see \cite{Chi}). Moreover if $\omega$ denotes the fundamental three-form of the related $\Gtwo$-structure then $2F^\psi=T^\psi\neg\omega$.
\end{rem}
We can then compute the explicit expression of $Q$ at ${s}=(g,\psi)$ in terms of $T^\psi$, $F^\psi$ and $D^\psi$.
\begin{prop}\label{prop1}
According to the decomposition $T_s\Cal{S}_1=\Cal{V}_s\oplus\Cal{H}_s$ the gradient $Q$ of $\Cal{E}$ at ${s}=(g,\psi)$ is given by $Q_v+Q_h$, where
\begin{enumerate}
\item $Q_v=\mathrm{div}^g\, T^\psi\cdot \psi+D^\psi$;
\item $(Q_h)_{ab}=(\mathrm{div}^g\,F^\psi)_{ab}^{\phantom{i}}+\frac{1}{2}(T^\psi*_gT^\psi)_{ab}-\frac{1}{2}|T^\psi|^2_g g_{ab}^{\phantom{a}}$.
\end{enumerate}
The terms $T^\psi$, $F^\psi$ and $D^\psi$ are defined in \eqref{T}, \eqref{F} and \eqref{D} whereas $\mathrm{div}^g$ denotes the divergence operator, which is the formal adjoint of $-\nabla^g$.
\end{prop}
\begin{proof}
First let us concern about the vertical part.
We can consider a variation of ${s}$ with $\dot{{s}}=(0,\dot{\psi})$: we are keeping the metric $g$ fixed. For simplicity let $T=T^\psi$ and $\nabla=\nabla^g$. Then
$$D\Cal{E}(\dot{{s}})=\int_M g(\dot{T},T)\vg.$$
Fix a orthonormal local frame $\left\{e_1,\dots,e_n\right\}$ near $x\in M$ such that $(\nabla_a e_b)_x=0$. Let us observe that, at $x$,
\begin{align*}
\dot{T}_{ab}&=<\nabla_a \dot{\psi},e_b\cdot \psi>+<\nabla_a\psi,e_b\cdot\dot{\psi}>,\\
&=\nabla_a<\dot{\psi},e_b\cdot \psi>-<\dot{\psi},\nabla_a e_b\cdot \psi>-<\dot{\psi}, e_b\cdot \nabla_a \psi>+<\nabla_a\psi,e_b\cdot\dot{\psi}>,\\
&=\nabla_a<\dot{\psi},e_b\cdot \psi>-2<\dot{\psi}, e_b\cdot \nabla_a \psi>.
\end{align*}
Then contracting this expression with $T^{ab}$ we obtain
\begin{align*}
\dot{T}_{ab}T^{ab}&=\nabla_a(<\dot{\psi},e_b\cdot \psi>)T^{ab}-2<\dot{\psi}, e_b\cdot \nabla_a \psi>T^{ab},\\
&=\nabla_a(<\dot{\psi},e_b\cdot \psi>T^{ab})-<\dot{\psi},e_b\cdot \psi> \nabla_a(T^{ab})-2<\dot{\psi}, e_b\cdot \nabla_a \psi>T^{ab},\\
&=\nabla_a(<\dot{\psi},e_b\cdot \psi>T^{ab})-<\dot{\psi},\nabla_a(T^{ab}e_b)\cdot \psi> -2<\dot{\psi}, T^{ab} e_b\cdot \nabla_a \psi>.
\end{align*}
The first term in the right side is the divergence of a globally defined vector field and then vanishes by taking the integral, so it turns out that
\begin{align*}
\int_M\dot{T}_{ab}T^{ab}\vg=-\int_M <\dot{\psi},\nabla_a(T^{ab}e_b)\cdot \psi+2T^{ab} e_b\cdot \nabla_a \psi>\vg,
\end{align*}
giving $Q_v(\dot{{s}})=\mathrm{div}^g\, T\cdot \psi + D^\psi$.\par
Now let us compute $Q_h$ evaluated at $\dot{{s}}=h\in \Cal{H}_{s}$ with ${s}=(g_0,\psi_0)$ in a given point $x\in M$. We consider $g_t$ (for simplicity we omit the dependence on $x$) to be the curve $g_0+th$ and ${s}_t$ its horizontal lifting with ${s}_0={s}$. With no loss of generality assume $t\in[0,1]$ and set $e^0=dt$. Following \cite{Am16} and adopting the notation therein we introduce the cylinder $C=[0,1]\times M$ equipped with the Riemannian metric $G=dt^2+g_t$. The cylinder inherits a spin structure from those of $M$ and $[0,1]$, so $G$ defines a spinor bundle $\Sigma^G$ on $C$. There are embeddings $\Sigma^{g_t}\subseteq\Sigma^G$ but the Clifford multiplications $\cdot$ on $\Sigma^{g_t}$  and $\cdot_G$ on $\Sigma^G$ do not agree:
$$X\cdot \psi= e_0\cdot_G X\cdot_G\psi,\quad \psi\in\Sigma^{g_t},\,X\in TM.$$
The introduction of $C$ allows us to characterize the horizontal lifting. Indeed the parallel transport with respect to $\nabla^G$ from $\Sigma^{g_0}$ to $\Sigma^{g_t}$ coincides with the parallel displacement along $g_t$ with respect to the natural connection.\par
The first variation of $\Cal{E}$ at ${s}$, evaluated at $\dot{{s}}$, is
\begin{align*}
D\Cal{E}_{s}(\dot{{s}})&=\frac{1}{2}\left.\frac{d}{dt}\right|_{t=0}\int_M |T^\psi|_g^2\vg,\\
&=\frac{1}{2}\int_M \left.\frac{d}{dt}\right|_{t=0}\left( |T^\psi|_g^2\right)\mathrm{vol}^{g_0}+ \frac{1}{2}|T^\psi|_{g_0}^2 \mathrm{Trace}_{g_0}(h)\mathrm{vol}^{g_0}.
\end{align*}
First we study the first term on the right side. 
\begin{align*}
\frac{1}{2}\left.\frac{d}{dt}\right|_{t=0} g(T^\psi,T^\psi)&=\left.g(\nabla^G_0T^\psi,T^\psi)\right|_{t=0}.
\end{align*}
Fix a orthonormal frame $\left\{e_1,\dots,e_n\right\}$ near $x\in M$ satisfying $\nabla^{g_0} e_j=0$ and extend it to $C$ by requiring $\nabla^G_{e_0}e_j=0$: then $\left\{e_0,\dots,e_n\right\}$ is a local orthonormal frame around $(0,x)\in C$ and $\nabla^G_0 (e_a\cdot\varphi)=e_a\cdot \nabla^G_0\varphi$ for any $a>0$ and spinor field $\varphi$. The covariant derivative of $(T^\psi)_{ab}=<\nabla^g_a\psi,e_b\cdot\psi>$  is
\begin{align*}
\nabla^G_0 T^\psi_{ab}&=<\nabla^G_0\nabla^g_a\psi,e_b\cdot\psi>+<\nabla^g_a\psi,e_b\cdot \nabla^G_0\psi>=<\nabla^G_0\nabla^g_a\psi,e_b\cdot\psi>,
\end{align*}
indeed $\nabla^G_0\psi$  vanishes since ${s}_t$ is horizontal. The expression $\left.\nabla^G_0\nabla^g_a\psi\right|_{t=0}$ was computed in \cite{Am16} (Formula 12) and it turns out to be
$$
\left.\nabla^G_0\nabla^g_a\psi\right|_{t=0}=W_a^{\phantom{a}k}\nabla^{g_0}_{k}\psi+\frac{1}{4}\sum_{i\neq j}(\nabla^{g_0}_i h^{\phantom{g_0}}_{ja}) e_i\cdot e_j\cdot\psi,
$$
where $W_a^{\phantom{a}k}=-(1/2)h_{a}^{\phantom{a}k}$ is the shape operator associated to $\left\{0\right\}\times M\subset C$. Therefore, writing $T$ for $T^\psi$ and $\nabla$ for $\nabla^{g_0}$,
\begin{align*}
\left.\nabla^G_0 T^{\phantom{a}}_{ab}\right|_{t=0}&=W_a^{\phantom{a}k}<\nabla_{k}\psi,e_b^{\phantom{j}}\cdot\psi>+\frac{1}{4}\sum_{i\neq j}(\nabla_i h^{\phantom{g_0}}_{ja}) <e_i\cdot e_j\cdot\psi,e_b^{\phantom{j}}\cdot\psi>,\\
&=-\frac{1}{2}h_a^{\phantom{a}k}T^{\phantom{a}}_{kb}+\frac{1}{4}\sum_{i\neq j}(\nabla_i^{\phantom{i}} h^{\phantom{g_0}}_{ja}) <e_i\cdot e_j\cdot \psi,e^{\phantom{b}}_b\cdot\psi>,\\
&=-\frac{1}{2}h_a^{\phantom{a}k}T^{\phantom{a}}_{kb}+\frac{1}{4}\sum_{i,j=1}^n(\nabla_i^{\phantom{i}} h^{\phantom{g_0}}_{ja}) <e_i\wedge e_j\cdot \psi,e^{\phantom{b}}_b\cdot\psi>
\end{align*}
Contracting the previous expression with $T^{ab}$, and recalling that $4F_{ija}=T_a^{\phantom{a}b}<e_i\wedge e_j\cdot \psi,e_b\cdot\psi>$ by \eqref{F}, it becomes
\begin{align*}
&-\frac{1}{2}h_a^{\phantom{a}k}T^{\phantom{a}}_{kb}T^{ab}+(\nabla^i h^{ja}_{\phantom{j}}) F_{ija}=-\frac{1}{2}h_a^{\phantom{a}k}T^{\phantom{a}}_{kb}T^{ab}+\nabla^i( h^{ja}_{\phantom{j}} F_{ija})-h^{ja}_{\phantom{j}} \nabla^i F_{ija}.
\end{align*}
Therefore
\begin{align}\label{eq1}
&\frac{1}{2}\int_M \left.\frac{d}{dt}\right|_{t=0}\left( |T^\psi|_g^2\right)\mathrm{vol}^{g_0}=-\int_M h_{ab}\left((\mathrm{div}^{g_0}\,F)^{ab}+1/2(T*_{g_0}T)^{ab}\right)\mathrm{vol}^{g_0}.
\end{align}
\par
Finally, since
\begin{equation*}
|T|_{g_0}^2\mathrm{Trace}_{g_0}(h)=T_{ab}T^{ab}h_{ij}g_0^{ij}=2g_0(h,|T|^2_{g_0}g_0),
\end{equation*}
it turns out that
\begin{equation}\label{eq2}
\frac{1}{4}\int_M |T|_{g_0}^2\mathrm{Trace}_{g_0}(h)\mathrm{vol}^{g_0}=\frac{1}{2}\int_M g_0(h,|T|^2_{g_0}g_0)\mathrm{vol}^{g_0}.
\end{equation}
Then, taking into account \eqref{eq1} and \eqref{eq2}, we find out that
$$(Q_h)_{ab}=\left(\mathrm{div}^{g_0} F\right)_{ab}+\frac{1}{2}(T*_{g_0}T)_{ab}-\frac{1}{2}|T|^2_{g_0}(g_0)_{ab}.$$
\end{proof}
\begin{rem}
Comparing the expression of $Q_v$ with that obtained in \cite{Ba17} for $n=7$, we find out that, in such setting, $D^\psi(\psi)=0$ for $|\psi|_g=1$ whence $Q_v=\mathrm{div}^g\,T^\psi\cdot \psi.$
\end{rem}
\begin{cor}\label{crit}
If ${s}=(g,\psi)\in\Cal{S}_1$ satisfies
\begin{equation}\label{pde}
\begin{cases}
\mathrm{div}^g\,T^\psi\cdot \psi+D^\psi=0,\\
 (\mathrm{div}^g\,F^\psi)_{(ab)}^{\phantom{i}}+(T^\psi*_gT^\psi)_{ab}-\frac{1}{2}|T^\psi|^2_g g_{ab}^{\phantom{a}}=0,
 \end{cases}
\end{equation}
then it has trivial energy tensor $T^\psi$, whence \eqref{pde} is trivially satisfied. Therefore any critical point of $\Cal{E}$ is an absolute minimizer. In particular, if $n=7$, critical points correspond to torsion-free $\Gtwo$-structures and vice-versa.  
\end{cor}
\begin{proof}
Let ${s}=(g,\psi)$ be a critical point, $\lambda>0$ and $B=B_{\lambda^2g}^g$. We define ${s}_\lambda$ to be $\tilde{B}s$. Then, locally, $s=[\sigma,v]$ and $s_\lambda=[\tilde{B}\sigma,v]$ for a local spin frame $\sigma$ and a $\Delta_n$-valued function $v$.  
On the other hand the Levi-Civita connections of $\lambda^2g$ and $g$ coincide, therefore
\begin{align*}
&\nabla^{\lambda^2g}_{Be_a} s_\lambda=(Be_a)(v^\alpha)\tilde{B}\sigma_{\alpha}+\frac{1}{2}v^\alpha\sum_{i<j}\lambda^2g(\nabla^g_{Be_a}Be_i,Be_j)Be_i\cdot Be_j\cdot \tilde{B}\sigma_\alpha,\\
&Be_b\cdot s_\lambda=v^\alpha Be_b\cdot \tilde{B}\sigma_\alpha.
\end{align*}
Thus
\begin{align*}
&<\nabla^{\lambda^2g}_{Be_a}s_\lambda,Be_b\cdot s_\lambda>\\
&=(Be_a)(v^\alpha)v^\beta<\tilde{B}\sigma_\alpha,Be_b\cdot\tilde{B}\sigma_\beta>+\\
&\;\;\;\frac{1}{2}v^\alpha v^\beta\sum_{i<j}\lambda^2g(\nabla^g_{Be_k}Be_i,Be_j)<Be_i\cdot Be_j\cdot \tilde{B}\sigma_\alpha, Be_b\cdot \tilde{B}\sigma_\beta>\\
&=\lambda^{-1}(e_a)(v^\alpha)v^\beta<\tilde{B}\sigma_\alpha,Be_b\cdot\tilde{B}\sigma_\beta>+\\
&\;\;\;\frac{1}{2}v^\alpha v^\beta\sum_{i<j}\lambda^{-1}g(\nabla^g_{e_k}e_i,e_j)<Be_i\cdot Be_j\cdot \tilde{B}\sigma_\alpha, Be_b\cdot \tilde{B}\sigma_\beta>,\\
&=\lambda^{-1}(e_a)(v^\alpha)v^\beta<\sigma_\alpha,e_b\cdot\sigma_\beta>+\\
&\;\;\;\frac{1}{2}v^\alpha v^\beta\sum_{i<j}\lambda^{-1}g(\nabla^g_{e_k}e_i,e_j)<e_i\cdot e_j\cdot \sigma_\alpha, e_b\cdot \sigma_\beta>,\\
&=\lambda^{-1}<\nabla^g_{e_a}s,e_b\cdot s>,
\end{align*}
where in the second equality we have used $B=\lambda^{-1}\mathrm{Id}$.
As consequence we obtain
$$T^{s_\lambda}=<\nabla^{\lambda^2g}_{Be_a}s_\lambda,Be_b\cdot s_\lambda>Be^a\otimes Be^b=\lambda <\nabla^g_{e_a}s,e_b\cdot s>e^{a}\otimes e^b=\lambda T^s.$$
It then follows that, being $\lambda^n\vg=\mathrm{vol}^{\lambda^2g}$,
$$\Cal{E}({s}_\lambda)=\lambda^{n+4}\Cal{E}({s}).$$
So, taking the derivative $\frac{d}{d\lambda}$ evaluated at $\lambda=1$, it must be
$$0=D\Cal{E}_{s}(\dot{{s}})=(n+4)\Cal{E}({s}),$$
whence $\Cal{E}({s})=0$.
\end{proof}
\section{Higher order energy tensors}\label{sec3}
Let $g\in\Cal{R}$. If $\psi\in\Gamma(M,\Sigma^g)$ is a non-zero spinor field of constant length then, for any $X\in\Gamma(M,TM)$, $\nabla^g_X\psi$ is orthogonal to $\psi$; that is $\nabla_X\psi\in\psi^\perp$. This follows from the compatibility condition of the spin connection $\nabla^g$ with the real metric $<.,.>_g$
$$0=\nabla^g_X|\psi|^2_g=2<\nabla_X^g\psi,\psi>_g.$$\par
On the other hand the skew-symmetry of the Clifford multiplication implies that
$$<X\cdot\psi,\psi>_g=-<\psi,X\cdot\psi>_g=-<X\cdot\psi,\psi>_g.$$
This means that we can embed $TM$ into $\psi^\perp\subset\Sigma^g$ as a real subbundle $TM^g_\psi$ by defining
$$TM^g_\psi=\left\{X\cdot\psi\;|\;X\in  TM\right\}.$$ 
Therefore $T^\psi$ identifies with the projection of $\nabla^g\psi\in T^*M\otimes\psi^\perp$ to $T^*M\otimes TM^g_\psi$. It is remarkable that $\psi^\perp=TM^g_\psi$ if and only if $n=3,7$; giving $T^\psi\cong \nabla^g\psi$. Generically, instead, the energy tensor $T^\psi$ fails to encode all the information about $\nabla^g\psi$.
\begin{cor}
For $n=3$ or $7$ the critical points of $\Cal{E}$ are couples $(g,\psi)$ with Ricci-flat $g$ and unit parallel $\psi$. 
\end{cor}
\par
However, for any $n\geq 3$, we can define some higher order tensors to catch all $\nabla^g\psi$. To this aim let us consider the family of tensors $\left\{T_r\right\}_{r=1,\dots,n}$ defined by
\begin{equation}\label{TT}
T_r(X_0,\dots,X_r)=<\nabla^g_{X_0}\psi,X_1\cdot\dotsc\cdot X_r\cdot \psi>_g,\quad X_j\in\Gamma(M,TM).
\end{equation}
Observe that, since $\psi$ never vanishes, the family of spinor fields $\left\{X_1\cdot\dotsc\cdot X_r\cdot \psi\right\}$ {generates}, over $\Cal{C}^\infty(M,\R)$, the space of section $\Gamma(M,\Sigma^g)$ (see Lemma \ref{lem1} for details). Therefore $\nabla^g \psi=0$ if and only if $T_r=0$ for any $1\leq r\leq n$. More precisely the following holds.
\begin{lem}\label{lem1}
The followings are equivalent:
\begin{enumerate}
\item $\nabla^g\psi=0$;
\item $T_r=0$ for any $1\leq r\leq n$;
\item $T_r=0$ for $r=n-1,n$.
\end{enumerate}
\end{lem}
\begin{proof}
It is clear that (1) implies (2) and (3). Let us prove the converse. Assume $T_r=0$ for any $r$ and take $x\in M$. For simplicity we suppose that $\Delta_n$ is irreducible. The Clifford algebra $\rm Cl$ attached to $(T_xM,g_x)$ is naturally isomorphic (as vector space) to the Grassmann algebra $\Lambda^\bullet T_xM$; in particular the algebra of the spin representation is generated by
$$\lambda\mathrm{Id},\quad X_1\wedge\dots\wedge X_r\cdot,\quad 1\leq r\leq n,\, X_j\in T_xM,\,\lambda\in\R,$$
or, equivalently, by
$$ X_1\cdot\dotsc\cdot X_r\cdot,\quad 1\leq r\leq n,\, X_j\in T_xM.$$
Therefore, by irreducibility, the span of $\left\{X_1\cdot\dotsc\cdot X_r\cdot\psi_x\right\}$ is the whole $\Sigma^g_x$, whence $\nabla^g\psi=0$. If $\Delta_n$ is reducible, that is $\Delta_n=\Delta_n^+\oplus\Delta_n^-$, then the claim follows by observing that the odd elements like $X_1\cdot\dotsc\cdot X_{2r+1}\cdot$ turn one irreducible module into the other.\par 
It remains to prove that (3) implies (2), so let us assume $T_r=0$ for $r=n-1,n$. Take $x\in M$ and let $X$ be a vector field which has constant length $1$ near $x$. Then, for $r\geq 3$ and thanks to the symmetries of the Clifford multiplication, we can contract each $T_r$ with $-X\otimes X$ so that the contracted $T_r$ equals $T_{r-2}$ near $x$. Therefore if $T_{n-1}$ and $T_n$ are zero everywhere then all the others must vanish as well.
\end{proof}
With the previous lemma in mind we define the family of functionals $\left\{\Cal{E}_r\right\}_{r=1,\dots,n}$ on $\Sigma$ by
\begin{equation}
\Cal{E}_{r}({s})=\frac{1}{2}\int_M |T_r|^2_g\vg,\quad {s}\in\Cal{S}_1.
\end{equation}
The following proposition then holds.
\begin{prop}
The critical points of $\Cal{E}_{n-1}+\Cal{E}_n$ are absolute minimezers, whence they are couples $(g,\psi)$ of Ricci-flat metrics and parallel spinor fields.
\end{prop}
\begin{proof}
We have seen in the proof of Corollary \ref{crit} that the action of $\widetilde{B}_{\lambda^2g}^g$ on $\Cal{E}_1$ is just the multiplication by $\lambda^{n+4}$. It is straightforward to check that each $\Cal{E}_r$ changes by a factor of $\lambda^{n+2+2r}$. Therefore any critical point ${s}$ has to be an absolute minimizer of $\Cal{E}_{n-1}+\Cal{E}_{n}$, so $T_{n-1}$ and $T_{n}$ identically vanish and the Lemma \ref{lem1} applies.
\end{proof}
With analogy to \eqref{F} and \eqref{D}, for ${s}=(g,\psi)\in\Cal{S}_1$ and $r=n-1,n$, we define
\begin{eqnarray}\label{FF}
4(F_r)^{\phantom{a}}_{ija}=(T_r)_{a}^{\phantom{a}a_1\dots a_r}<e_i\wedge e_j\cdot \psi,e_{a_1}\cdot\dotsc\cdot e_{a_r}\cdot \psi>_g,
\end{eqnarray}
and
\begin{equation}\label{DD}
D_{r}=2T_r^{aa_1\dots a_r}e_{a_1}\cdot\dotsc\cdot e_{a_r} \cdot \nabla_a \psi, 
\end{equation}
where $\left\{e_1,\dots,e_n\right\}$ is a $g$-orthonormal frame.
\begin{prop}\label{gradient}
The gradient $\Cal{Q}$ of $\Cal{E}_{n-1}+\Cal{E}_n$ at ${s}=(g,\psi)\in\Cal{S}_1$, according to the decomposition $\Cal{V}_{s}\oplus\Cal{H}_{s}$ and formulas \eqref{TT}, \eqref{FF} and \eqref{DD}, is given by $\Cal{Q}_v+\Cal{Q}_h$, where
\begin{align}\label{Q}
\Cal{Q}_v=&\mathrm{div}^g\,T_{n-1}\cdot\psi+\mathrm{div}^g\,T_{n}\cdot\psi+D_{n-1}+D_n,\\
(\Cal{Q}_h)_{ab}=&(\mathrm{div}^g\, F_{n-1})_{ab}+(\mathrm{div}^g\, F_n)_{ab}+\\
&\frac{1}{2}(T_{n-1}*_gT_{n-1})_{ab}+\frac{1}{2}(T_{n}*_gT_{n})_{ab}-\frac{1}{2}\left(|T_{n-1}|^2_g+|T_n|_g^2\right)g_{ab}.\nonumber
\end{align}
\end{prop}
\begin{proof}
The proof is analogous to that of Proposition \ref{prop1}, keeping attention to the new definitions of $T^\psi$, $F^\psi$ and $D^\psi$.
\end{proof}
\begin{prop}
The Hessian  of $\mathcal{E}_{n-1}+\mathcal{E}_n$ at a critical point $s$ is positive semi-definite.
\end{prop}
\begin{proof}
Let $\dot{s}=(h,\varphi)$ and take a positive real number $\varepsilon>0 $such that $g+th$ is positive definite for $|t|<\varepsilon$. As usual we introduce the Cylinder $(-\varepsilon,\varepsilon)\times M$ equipped with the Riemannian metric $dt^2+(g+th)$ and the Levi-Civita connection $\nabla^C$. Set $e_0=\frac{d}{dt}|_{\left\{0\right\}\times M}$. Then we observe that the first variation  of each $\Cal{E}_r$ at $s$, evaluated at $\dot{s}$, is
$$\int_M g\left(\nabla^C_{e_0}T_r,T_r\right)+\frac{1}{4}|T_r|^2_g\mathrm{Trace}_g(h)\vg.$$
Finally, since any critical point has to satisfy $T_r=0$, we derive that
$$D^2\left(\Cal{E}_{n-1}+\Cal{E}_n\right)_s(\dot{s},\dot{s})=\int_M |\nabla^C_{e_0}T_{n-1}|_g^2+|\nabla^C_{e_0}T_{n}|^2_g\vg,$$
 and the proposition follows.
\end{proof}
\begin{lem}\label{lemsymb}
Let $\xi$ be a tangent vector. The principal symbol of $D\Cal{Q}_{s}$, the linearisation of $\Cal{Q}$ at a section ${s}$, at $\xi$ is not invertible and its kernel is
\begin{equation}\label{diff}
\left\{(\dot{g},\dot{\psi})=(2\xi\odot v,-1/4\xi\wedge v\cdot \psi)\;|\; v\in TM\right\}.
\end{equation}
\end{lem}
\begin{proof}
We have to compute the linearisation of $P_v=\mathrm{div}^g\,T_n\cdot \psi+\mathrm{div}^g\,T_{n-1}\cdot\psi$, as it expresses the leading order term of $\Cal{Q}_v$. First assume that the variation is along the vertical directions. A straightforward computations shows that, using orthonormal frames,
\begin{align*}
DP_v(\dot{\psi})=&<\nabla^a\nabla_a\dot{\psi},e_{a_1}\cdot\dotsc\cdot e_{a_{n}}\cdot\psi>_g e^{a_1}\cdot\dotsc\cdot e^{a_n}\cdot\psi\\
&+<\nabla^a\nabla_a\dot{\psi},e_{b_1}\cdot\dotsc\cdot e_{b_{n-1}}\cdot\psi>_g e^{b_1}\cdot\dotsc\cdot e^{b_{n-1}}\cdot\psi+l.o.t.
\end{align*}
From now on let us write $\psi_\alpha$ and $\psi_\beta$ for $e_{a_1}\cdot\dotsc\cdot e_{a_n}\cdot\psi$ and $e_{b_1}\cdot\dotsc\cdot e_{b_{n-1}}\cdot\psi$, where $\alpha=(a_1\dots,a_n)$ and $\beta=(b_1,\dots,b_n)$ are multi indices. The set $\left\{\psi_\alpha,\psi_\beta\right\}$ is a generating set of the spin module, but not a basis, at any point of $M$.\par
The principal symbol of $DP_v$ at some non-zero tangent vector $\xi$ is  
$$\sigma_\xi(DP_v)(\dot{\psi})=|\xi|^2\left(\sum_\alpha<\dot{\psi},\psi_\alpha>\psi_\alpha+\sum_\beta<\dot{\psi},\psi_\beta>\psi_\beta\right),$$
which, contracted with $\dot{\psi}$, gives
$$<\sigma_\xi(DP_v)(\dot{\psi}),\dot{\psi}>=|\xi|^2\sum_{\gamma=\alpha,\beta}<\dot{\psi},\psi_\gamma>^2.$$\par
In general, for arbitrary variations, we have (Lemma 4.12 in \cite{Am16})
$$\frac{d}{dt}\nabla_a\psi=\sum_{i\neq j}\frac{1}{4}(\nabla_i \dot{g}_{aj})e_i\cdot e_j\cdot \psi+\nabla_a\dot{\psi},$$
giving 
$$\frac{d}{dt}{\mathrm{div}^g\,{T}_{n}}=\sum_\alpha<1/4\sum_{i,j}(\nabla^a\nabla_i \dot{g}_{aj})e_i\wedge e_j\cdot \psi+\nabla^a\nabla_a\dot{\psi},\psi_\alpha>\psi_\alpha+l.o.t.,$$
and a similar expression for $T_{n-1}$. The principal symbol turns out to be
\begin{equation}\label{lin1}
DP_v=\sum_\alpha<1/4\xi\wedge (\xi\neg \dot{g})\cdot \psi+|\xi|^2\dot{\psi},\psi_\alpha>\psi_\alpha.
\end{equation}
Taking into account the expression of $T_{n-1}$ and contracting with $\dot{\psi}$ we obtain
$$<\sigma_\xi(DP_v)(\dot{g},\dot{\psi}),\dot{\psi}>=\sum_{\gamma=\alpha,\beta} 1/4<\xi\wedge (\xi\neg \dot{g})\cdot \psi,\psi_\gamma><\dot{\psi},\psi_\gamma>+\xi^2<\dot{\psi},\psi_\gamma>^2.$$\par
On the other hand we have to evaluate the linearisation of
$$P_h=\mathrm{div}^g\, (F_n)_{ab}+\mathrm{div}^g\, (F_{n-1})_{ab},$$
which is the leading order term of $\Cal{Q}_h$. We first consider $F=F_n$; for $F_{n-1}$ counts are analogous. We have 
\begin{align*}
\nabla^aF_{aij}=&1/4\sum_\alpha<e_a\wedge e_i\cdot \psi,(\nabla^aT_{j\alpha})\psi_\alpha>+l.o.t.
\end{align*}
In analogy with the divergence of $T$ it turns out that
$$\frac{d}{dt}\nabla^aT_{j\alpha}=<1/4\sum_{p,q}\nabla^a\nabla_p\dot{g}_{jq}e_p\wedge e_q\cdot \psi+\nabla^a\nabla_j\dot{\psi},\psi_\alpha>+l.o.t.$$
Thus
\begin{align*}
&\frac{d}{dt}\nabla^aF_{aij}=\\
&1/4\sum_\alpha<1/4\sum_{p,q}\nabla^a\nabla_p\dot{g}_{jq}e_p\wedge e_q\cdot \psi+\nabla^a\nabla_j\dot{\psi},\psi_\alpha><e_a\wedge e_i\cdot\psi,\psi_\alpha>+\\
&+l.o.t.
\end{align*}
Taking the principal symbol we find
\begin{align}\label{lin2}
\sigma_\xi(DP_h)=&1/16\sum_\alpha<\xi\wedge (e_j\neg\dot{g})\cdot \psi,\psi_\alpha><\xi\wedge e_i \cdot \psi,\psi_\alpha>e^i\otimes e^j+\\\nonumber
&1/4\sum_\alpha \xi_j<\dot{\psi},\psi_\alpha><\xi\wedge e_i\cdot\psi,\psi_\alpha>e^i\otimes e^j.
\end{align}
Finally, contracting this expression with $\dot{g}$, it results
\begin{align*}
&1/16\sum_j\sum_\alpha<\xi\wedge (e_j\neg\dot{g})\cdot \psi,\psi_\alpha>^2+\\
&1/4\sum_\alpha <\dot{\psi},\psi_\alpha><\xi\wedge (\xi\neg\dot{g})\cdot\psi,\psi_\alpha>.
\end{align*} 
Summarizing we have found
\begin{align}\label{symb}
<\sigma_\xi(D\Cal{Q}_{s})(\dot{g},\dot{\psi}),(\dot{g},\dot{\psi})>=&
\sum_{\gamma=\alpha,\beta}|\xi|^2<\dot{\psi},\psi_\gamma>^2\\\nonumber&+1/16\sum_j<\xi\wedge (e_j\neg\dot{g})\cdot \psi,\psi_\gamma>^2\\\nonumber
&+1/2 <\dot{\psi},\psi_\gamma><\xi\wedge (\xi\neg\dot{g})\cdot\psi,\psi_\gamma>.
\end{align}\par
We claim that the kernel of the symbol coincides with the vector space $N$ defined in \eqref{diff}.
To prove the claim take $(\dot{g},\dot{\psi})$ in the kernel. Then, by \eqref{lin1}, it must be
$$1/4\xi\wedge (\xi\neg \dot{g})\cdot \psi+|\xi|^2\dot{\psi}=0,$$
so, defining $v$ to be $|\xi|^{-2}\xi\neg \dot{g}$, we obtain the second equation. Then, since $\dot{\psi}=-1/2|\xi|^{-2}\xi\wedge (\xi\neg \dot{g})\cdot \psi$, \eqref{lin2} is equivalent to
\begin{align}\label{contraction}
\sum_\gamma\left(<\xi\wedge(e_j\neg \dot{g})\cdot\psi,\psi_\gamma>-|\xi|^{-2}\xi_j<\xi\wedge (\xi\neg\dot{g})\cdot\psi,\psi_\gamma>\right)<\xi\wedge e_i\cdot \psi,\psi_\gamma>=0,
\end{align}
for any $1\leq i,j\leq n$. Assume $|\xi|=1$ and choose a orthonormal frame with $\xi=e^1$. Then we have
$$\sum_\gamma\left(<e_1\wedge(e_j\neg \dot{g})\cdot\psi,\psi_\gamma>\right)<\xi\wedge e_i\cdot \psi,\psi_\gamma>=0,\quad 1\leq i\leq n,\; 1<j\leq n.$$
Multiplying for $\dot{g}_{ij}$ and summing over $i$ gives
$$\sum_\gamma <e_1\wedge(e_j\neg \dot{g})\cdot\psi,\psi_\gamma>^2=0,\quad 1<j\leq n,$$
or equivalently
$$e_1\wedge(e_j\neg \dot{g})=0,\quad 1<j\leq n,$$
that is $\dot{g}=\sum_j\lambda_j\left(e_1\otimes e_j+e_j\otimes e_1\right)$ for $\lambda_j\in\R$. Finally, since $v=e_1\neg \dot{g}$, it follows that $v=\sum_j \lambda_je_j$ and $\dot{g}=2\xi\odot v$; this proves $\mathrm{Ker}\subseteq N$. To show $N\subseteq \mathrm{Ker}$ take $(\dot{g},\dot{\psi})\in N$ and $\xi=e_1$ as above. Then the vanishing of the vertical part \eqref{lin1} is evident, as well as the left side of \eqref{contraction} for $j>1$ and $1\leq i \leq n$. For $j=1$ the former becomes
$$\sum_\gamma\left(<\xi\wedge(\xi\neg \dot{g})\cdot\psi,\psi_\gamma>-<\xi\wedge (\xi\neg\dot{g})\cdot\psi,\psi_\gamma>\right)<\xi\wedge e_i\cdot \psi,\psi_\gamma>,\quad 1\leq i\leq n,$$
which is manifestly zero.
\end{proof}
\begin{lem}
Let $\xi$ be a tangent vector. The principal symbol of $D\Cal{Q}_{s}$ at $\xi$ is positive definite on the orthogonal complement of its kernel.
\end{lem}
\begin{proof}
In the Lemma \ref{lemsymb} we have seen that the kernel is given by \eqref{diff}. Let $(\dot{g},\dot{\psi})$ be a non-zero element orthogonal to it. We claim that $\xi\wedge (\dot{g}\neg \xi)=0$. Otherwise, putting $v=\dot{g}\neg \xi$, we would have $\dot{g}(\xi,v)\neq 0$, whence $g(\dot{g},\xi\odot v)\neq 0$: a contradiction. Then setting $\xi\wedge (\dot{g}\neg \xi)=0$ in \eqref{symb} we find out that
\begin{align*}
<\sigma_\xi(D\Cal{Q}_{s})(\dot{g},\dot{\psi}),(\dot{g},\dot{\psi})>=&
\sum_{\gamma=\alpha,\beta}|\xi|^2<\dot{\psi},\psi_\gamma>^2+1/16\sum_j<\xi\wedge (e_j\neg\dot{g})\cdot \psi,\psi_\gamma>^2.
\end{align*}
Finally choosing a orthonormal frame in which $\dot{g}$ has diagonal form we see that $\sigma_\xi(D\Cal{Q}_{s})(\dot{g},\dot{\psi})>c|\xi|^2|\dot{\psi}|^2|\dot{g}|^2$ for some $c>0$. \end{proof}
The situation here is identical to that found in \cite{Am16}, or in other settings as in \cite{Gri2}, \cite{Ham2} and \cite{Vez}. Precisely the operator $D\Cal{Q}_{s}$ turns out to be strongly elliptic in certain directions only. The existence of degenerate directions is due to the presence of symmetries of the energy functional, and then of its gradient. Indeed the universal covering group $\widetilde{\mathrm{Diff}}^0(M)$ of the connected component of the full diffeomorphism group $\rm{Diff}^0(M)$ has a natural action on $\widetilde{\Cal{P}}$, hence on $\Sigma_1$, which preserves the fibers and lifts the action on $S^2_+T^*M$; so the corresponding tangent action must generates a non-trivial space of degenerate directions; actually all them (as we are going to see).\par
Given a vector field $X\in\Gamma(M,TM)$ with flow $f_t$ we consider the group action $(f_t)_*$ on $\Cal{P}$ and  $(\tilde{f}_t)_*$ on $\Cal{P}$, obtained by lifting $(f_t)_*$ from the identity map. Then $X$ acts on $\Sigma_1$ by
\begin{equation*}
\Cal{L}_Xs=\left.\frac{d}{dt} \left[(\tilde{f}_t)_*\sigma,v\right] \right|_{t=0},\quad s=[\sigma,v].
\end{equation*}
\begin{lem}[\cite{BG92}]
According to the decomposition $T_s\Sigma_1 =H_s\oplus V_s$ it turns out that
$$\Cal{L}_Xs=\left(L_Xg,\nabla^g_X\psi-\frac{1}{4}dX\cdot\psi\right),\quad s=(g,\psi),$$
where $L_X$ and $d X$ denote the Lie derivative with respect to $X$ and the exterior derivative of the $g$-dual of $X$ respectively.
\end{lem} 
A simple computation shows that the symbol of the linearisation of $X\mapsto \Cal{L}_Xs$ at a tangent vector $\xi$ is
\begin{equation}\label{symbolderivative}
\left(2\xi\odot\dot{X},-\frac{1}{4}\xi\wedge\dot{X}\cdot \psi\right),\quad \dot{X}\in TM,
\end{equation}
which is exactly the kernel of $\sigma_\xi D_s\Cal{Q}$, described by \eqref{diff}.
\par
In these circumstances we can apply the Hamilton and De Turk technique (\cite{DeT83}) to break the (spin) diffeomorphism equivariance of $\Cal{Q}$, and then define a genuine (non-linear) strongly elliptic second order differential operator. This operator, obtained by adding a suitable Lie derivative to $\Cal{Q}$, will allow us to prove the short-time existence of the gradient flow. Concretely we proceed as follows. Keeping in mind Formula \eqref{symbolderivative} let $s$ be a section of $\Sigma_1$ and let us suppose $X$ to be a vector field depending linearly on the first derivatives of $s=(g,\psi)$. We demand whether there exists such a $X$ for which the linear operator $\Cal{X}=D\Cal{Q}_s+D\Cal{L}_Xs$ has positive definite symbol at $\xi$. By \eqref{symbolderivative} we have
\begin{align*}
<\sigma_\xi D\Cal{L}_X(\dot{s}),\dot{s}>=&2\dot{g}(\xi,\dot{X})-\frac{1}{4}<\xi\wedge \dot{X}\cdot\psi,\dot{\psi}>,\\
\geq&2\dot{g}(\xi,\dot{X})-\frac{1}{4}|\xi\wedge \dot{X}\cdot \psi||\dot{\psi}|.
\end{align*} 
On the other hand, by \eqref{symb}, it is
\begin{align*}
<\sigma_\xi(D\Cal{Q})(\dot{s}),\dot{s}>=&
\sum_{\gamma}|\xi|^2<\dot{\psi},\psi_\gamma>^2\\\nonumber&+1/16\sum_j<\xi\wedge (e_j\neg\dot{g})\cdot \psi,\psi_\gamma>^2\\\nonumber
&+1/2 <\dot{\psi},\psi_\gamma><\xi\wedge (\xi\neg\dot{g})\cdot\psi,\psi_\gamma>.
\end{align*}
Note that each $\psi_\gamma$ has norm one. The previous sum is then greater or equal than
$$\sum_\gamma|\xi|^2<\dot{\psi},\psi_\gamma>^2+\frac{1}{16}\sum_{\gamma}\sum_j<\xi\wedge(e_j\neg \dot{g})\cdot\psi,\psi_\gamma>^2-\frac{1}{2}\sum_\gamma|\xi\wedge(\xi\neg\dot{g})\cdot \psi||\dot{\psi}|,$$
where the third term comes from the Cauchy-Schwarz inequality.
Therefore, setting $c_n=\sum_\gamma1=n^n+n^{n-1}$, we find out that 
\begin{align*}
<D\Cal{X}(\dot{s}),\dot{s}>\geq& \sum_\gamma|\xi|^2<\dot{\psi},\psi_\gamma>^2+\frac{1}{16}\sum_{\gamma}\sum_j<\xi\wedge(e_j\neg \dot{g})\cdot\psi,\psi_\gamma>^2+\\
&2\dot{g}(\xi,\dot{X})-\frac{1}{4}\left(|\xi\wedge\dot{X}\cdot\psi|-2c_n|\xi\wedge(\xi\neg \dot{g})\cdot\psi|\right)|\dot{\psi}|.
\end{align*}
This expression is certainly positive\footnote{ Recall that $\left\{\psi_\gamma\right\}_{\gamma}$ generates the spin module.} provided that $2c_n\xi\neg\dot{g}=\dot{X}$: indeed we would have $\dot{g}(\xi,\dot{X})=2c_n|\xi\neg \dot{g}|^2$ and $|\xi\wedge\dot{X}\cdot\psi|-2c_n|\xi\wedge(\xi\neg\dot{g})\cdot\psi|=0$; also note that the sum $|\xi\neg \dot{g}|^2+\sum_\gamma\sum_j<\xi\wedge(e_j\neg\dot{g})\cdot \psi,\psi_\gamma>^2$ is zero if and only if $\dot{g}=0$.\par
Finally, in order to show the existence of one vector field satisfying $2c_n\xi\neg \dot{g}=\dot{X}$, we choose an arbitrary Riemannian metric $\bar{g}$ and then define
$$X(s)=2c_n\mathrm{div}^{\bar{g}}g.$$
We have then proved the following lemma.
\begin{lem}\label{posdef}
Let $\bar{g}$ be a Riemannian metric and let $\xi$ be a tangent vector. Further let us consider the differential operator
\begin{equation*}
\Cal{S}_1\ni s\mapsto\mathrm{Q}(s)\in\Cal{S},\quad \mathrm{Q}(s)=\Cal{Q}(s)+\Cal{L}_{X(s)}s,
\end{equation*}
where 
$$X(s)=2n^{n-1}(n+1)\mathrm{div}^{\bar{g}}\Theta(s).$$
Then the principal symbol of the linearised second order differential operator $D\mathrm{Q}_s$ at $\xi$ is positive definite. In other words the operator $D\mathrm{Q}_s$ is strongly elliptic for any $s\in\Cal{S}_1$.  
\end{lem}
Consequently we obtain the short-time existence of the relative flow:
\begin{lem}\label{existence1}
Let $s\in\Cal{S}_1$. Then there exists a positive real number $\epsilon$ such that the De Turk flow
\begin{equation}\label{DeTurk}
\begin{cases}
\dot{s}_t=\mathrm{Q}(s_t),\\
s_0=s,
\end{cases}
\end{equation}
has a unique solution, living in $\Cal{S}_1$, for $t\in[0,\epsilon)$. 
\end{lem}
\begin{proof}
The proof exactly realizes as in \cite{BX}, \cite{Gri} and \cite{Vez} (which are all inspired by \cite{Ham2}). First we observe that, despite $\Cal{S}_1$ is not a vector bundle over $M$, it is a Fréchet manifold. In particular we can trivialise an open neighbourhood $\Cal{U}^0$ of $s=(g,\psi)$ and identify it, thanks to the horizontal distribution, to a neighbourhood of the zero section of $W=S^2T^*M\oplus\psi^\perp$. Then we can consider $\mathrm{Q}$ as a map between vector bundles, hence Fréchet spaces.
Let $\Cal{U}$ denote the space of time-dependent sections lying in $\Cal{U}^0$, and choose $T>0$.\par 
The main idea of the proof is to consider the locally defined map between Fréchet spaces 
\begin{align}\label{Ham}
F:\;\Cal{U}\ni s_t\mapsto (\dot{s}_t-\mathrm{Q}(s_t),s_0)\in\Gamma(M\times[0,T],W)\times \Gamma(M,W).
\end{align}
The spaces $\Gamma(M\times[0,T],W)$ and $\Gamma(M,W)$ can be equipped with a tame Fréchet structure for which $F$ turns to be a smooth tame Fréchet map (see \cite{Ham1}).
Furthermore, for any $u\in\Cal{U}$, the tangent map $DF_{u}$ represents a linear parabolic PDEs system by Lemma \ref{posdef}. The standard PDE theory then ensures the existence, uniqueness and $L^2_k$-estimates for its solutions. In the language of tame Fréchet spaces these properties are equivalent to the smooth tameness of the family of inverses $DF_u^{-1}$ (\cite{BX}). These are exactly the hypothesis of the inverse function theorem for Fréchet spaces (\cite{Ham1}), which therefore applies to show the local invertibility and smooth tameness of $F^{-1}$.\par  
Now we choose a time-dependent section $\bar{s}_t$ such that its formal power series at $t=0$ solves the initial value flow equation at that time. Then we set $(\bar{f}_t,s_0)=F(\bar{s}_t)$ and choose $\epsilon>0$. Since any time derivative of $\bar{f}_t$ vanishes at $t=0$ we can define another smooth section ${f}_t$ satisfying ${f}_t=0$ for $t<\epsilon$ and ${f}_t=\bar{f}_{t-\epsilon}$ for $t\geq \epsilon$. Provided that $\epsilon$ is small the section ${f}_t$ will be close to $\bar{f}_t$ with respect to the Fréchet structure. We can then define the section $s_t=F^{-1}(f_t,s_0)$. By definition this section satisfies $\dot{s}_t-\mathrm{Q}(s_t)=f_t$, which identically vanishes in $[0,\epsilon)$. Consequently $s_t$ is the unique solution to Equation \eqref{DeTurk} for $t\in[0,\epsilon)$.
\end{proof}
Finally we can derive the short-time existence of the gradient flow.
\begin{prop}
Let ${s}\in\Cal{S}_1$. Then there exists a positive real number $\epsilon$ such that the gradient flow
\begin{equation}\label{gradientflow}
\begin{cases}
\dot{{s}}_t=\Cal{Q}({s}_t),\\
{s}_0={s},
\end{cases}
\end{equation}
has a unique solution, living in $\Cal{S}_1$, for $t\in[0,\epsilon)$.
\end{prop}
\begin{proof}
Let us consider a solution $\bar{s}_t$ of the initial value problem \eqref{DeTurk} with $\bar{s}_0=s$, which exists by Lemma \ref{existence1}. We can integrate the one-parameter family of vector fields $X(t)=2n^{n-1}(n+1)\mathrm{div}^{\bar{g}}\Theta(\bar{s}_t)$ to a one-parameter family of diffeomorphisms $f_t$ of $M$ by solving the standard ODE
\begin{equation*}
\begin{cases}
\dot{f}_t=-X(t)_{f_t},\\
f_0=\mathrm{Id}_M.
\end{cases}
\end{equation*}
We claim that the section $s_t=(\tilde{f}_t)_*(\bar{s}_t)$ solves \eqref{gradientflow}. To prove this let us compute the time derivative of $s_t$, since the initial condition is straightforwardly satisfied. Taking into account the definition of $\mathrm{Q}$ and the diffeomorphism equivariance of $\Cal{Q}$ we get
\begin{align*}
\frac{d}{dt}s_t=&\frac{d}{dt}(\tilde{f}_{t})_*\bar{s}_t+(\tilde{f}_t)_*\frac{d}{dt}\bar{s}_t=(\tilde{f_t})_*\left(\Cal{L}_{X_t}\bar{s}_t +\mathrm{Q}(\bar{s}_t)\right),\\
=&(\tilde{f}_t)_*(\Cal{Q}(\bar{s}_t))=\Cal{Q}(s_t),
\end{align*}
and the claim follows.\par
Finally we can affirm that $s_t$ is the unique solution because, by reversing the procedure, any other solution $\hat{s}_t$ of \eqref{gradientflow} gives rise to a solution of \eqref{DeTurk}, which has to coincide with $\bar{s}_t$. Then also $\hat{s}_t$ will coincide with $s_t$.
\end{proof}
\section{Seven dimensions}\label{sec4}
For the rest of the section let us assume $n=7$. Recall that a $\Gtwo$-structure on a seven-dimensional spin manifold $M$ is a reduction $\Cal{F}$ of $\Cal{P}$ with structure group $\Gtwo$, the connected compact simple Lie group with Lie algebra $\mathfrak{g}_2$. This datum is equivalent (see \cite{Bry} and many others) to the choice of a Riemannian metric $g$ and a unit spinor field $\psi\in\Gamma(M,\Sigma^g)$ (up to a sign). Therefore $\Cal{S}_1/\mathbb{Z}_2$ can be interpreted as the space of $\Gtwo$-structures over $M$.\par
It turns out that a $\Gtwo$-structure is uniquely determined by a certain three-form $\omega\in\Omega^3(M)$, called the \emph{fundamental three-form} of the structure, defined as
\begin{equation}\label{G2form}
\Omega(g,\psi)_{abc}=\omega_{abc}=<e_a\cdot e_b\cdot e_c\cdot \psi,\psi>_g;
\end{equation}
often, to highlight the dependence of the metric $g$ by the three-form $\omega$, we write $g=g_\omega$ and we say that $g$ is compatible with $\omega$.
The subset $\Lambda^3_+$ of $\Lambda^3T^*M$ of those elements that are pointwise $\GL$-equivalent to \eqref{G2form} is open and therefore the subspace $\Omega^3_+$ of the three-forms with range $\Lambda^3_+$, called the space of \emph{positive} three-forms, is also open.\par
The map $\Omega$ then is a two-fold covering of $\Sigma_1$ onto $\Lambda^3_+$. In particular we can read the vertical and horizontal distributions on $\Lambda^3_+$. Precisely let $s=(g,\psi)\in\Sigma_1$ over $\omega\in\Lambda^3_+$ at some point $x\in M$. The vertical space $V_s$ at $s$ is $\left\{(g,\varphi)\;|\;<\varphi,\psi>_g=0\right\}$. Since $\psi^\perp_x=(TM_\psi^g)_x$ for any $\varphi\in\psi^\perp$ there exists a unique vector $X_\varphi$ of length $1$ such that $\varphi=X_\varphi\cdot \psi$ and vice-versa. Thus the image of $V_s$ turns out to be
$$\Lambda^3_7(\omega)=\left\{X\neg \omega\;|\; X\in (T_xM)_1\right\};$$
indeed the variations of $\omega$ in $\Lambda^3_7(\omega)$ do not perturb the metric and it is well known to be totally integrable with leaves given by structures with the same underlying metric (see \cite{Bry}, Formula 3.5, or \cite{Ba17}, Lemma 3.2).\par
On the other hand there is an obvious candidate for the image of the horizontal distribution. Indeed there exists a $\Gtwo$-equivariant injective map (see for instance \cite{Bry})
$$S^2T^*_xM\ni h\mapsto i_\omega(h)\in\Lambda^3T^*_x M,\quad i_\omega(h)^{\phantom{j}}_{abc}=h^{\phantom{j}}_{[a|i|}\omega^{i}_{\phantom{i}bc]},$$
whose image $\Lambda^3_{28}(\omega)$ is complementary to $\Lambda^3_7(\omega)$. Actually this family corresponds to the horizontal distribution.
\begin{lem}\label{hdistribution}
Under the correspondence \eqref{G2form} the horizontal distribution of Bourguignon and Gauduchon on $\Sigma_1$ identifies with the distribution on $\Lambda^3_+$ generated by
$$\Lambda^3_{28}(\omega)=\left\{i_\omega(h)\;|\;h\in S^2T^*_xM\right\}.$$
In particular this distribution is not integrable. Furthermore if $(g_t,\psi_t)$ is a horizontal curve in $\Sigma_1$ then $\omega_t=\Omega(g_t,\psi_t)$ has velocity $3/2 i_{\omega_t}(\dot{g}_t)$.
\end{lem}
\begin{proof}
Let $x\in M$, $g$ be a metric on $T_xM$, $\left\{e_1,\dots,e_7\right\}$ be a orthonormal frame at $x$. Further let us assume $\psi\in\Sigma^g_x$ to be a unit spinor. Take $h\in S^2T^*_xM$ and define $g_t=g+th$ for small values of $t\in\R $. We apply the generalized cylinder calculus as in the proof of Proposition \ref{prop1}. Let $(g_t,\psi_t)$ be the horizontal lifting of $g_t$ on $\Sigma_1$. Denote by $e_{0,t}$ the unit tangent vector to the cylinder given by the time derivative. As usual we extend $e_{1},\dots,e_{7}$ to vector fields $e_{1,t},\dots,e_{7,t}$ on the time line by requiring $\nabla^G_{e_{0,t}}e_{a,t}=0$.\par
The fundamental three-form $\omega_t$ relative to $(g_t,\psi_t)$ is given by
\begin{align*}
(\omega_t)_{abc}=&<e_{a,t}\cdot e_{b,t}\cdot e_{c,t}\cdot \psi_t,\psi_t>=<e_{0,t}\cdot_Ge_{a,t}\cdot_Ge_{b,t}\cdot_Ge_{c,t}\cdot_G\psi_t,\psi_t>.
\end{align*}
Taking the covariant derivative we find $\nabla^G_{e_{0,0}}\omega_t=0$. This expression can be expanded to get
\begin{align*}
0=\nabla^G_{e_{0,0}}(\omega_t)_{abc}=e_{0,0}^{\phantom{j}}(\omega_{abc}^{\phantom{j}})-\omega_{pbc}^{\phantom{j}}\Gamma_{a 0}^p-\omega_{aqc}^{\phantom{j}}\Gamma_{b 0}^q-\omega_{abr}^{\phantom{j}}\Gamma_{c 0}^r.
\end{align*}
Since each Christoffel symbol $\Gamma_{i0}^j$ of the Levi-Civita connection $\nabla^G$ is $1/2\,h_{i}^{\phantom{i}j}$ we can write
\begin{align*}
\frac{d}{dt}\left.\omega_{abc}\right|_{t=0}=&\frac{1}{2}\left(h_{a}^{\phantom{a}p}\omega_{pbc}^{\phantom{j}}+h_{b}^{\phantom{a}q}\omega_{aqc}^{\phantom{j}}+h_{c}^{\phantom{a}r}\omega_{abr}^{\phantom{j}}\right)=\frac{3}{2}i_{\omega}(h)_{abc}.
\end{align*}
The lemma then follows.
\end{proof}
As consequence we obtain the following corollary.
\begin{cor}\label{corode}
Let $\omega\in\Omega^3_+(M)$ the fundamental three-form of a $\Gtwo$-structure with underlying metric $g$. Further let us assume that $\left\{g_t\right\}_{t\in I}$ is a smooth curve of metrics starting at $g_0=g$. Then the solution to the following ODE
\begin{equation}\label{ode}
\begin{cases}
\dot{\omega}_t=\frac{3}{2}i_{\omega_t}(\dot{g}_t),\\
\omega_0=\omega,
\end{cases}
\end{equation}
exists as long as $g_t$, belongs to $\Omega^3_+(M)$ and induces the Riemannian metric $g_t$. Moreover if $g_t$ satisfies
\begin{equation}\label{flat}
B_{g_{t_1}}^g\circ B_{g_{t_2}}^g=B_{g_{t_2}}^g\circ B_{g_{t_1}}^g,\quad t_1,t_2\in I.
\end{equation}
then the value of $\omega_t$ at $\bar{t}\in I$ is independent on the path $g_t$ satisfying \eqref{flat} and having values $g_0$ at $0$ and $g_{\bar{t}}$ at $\bar{t}$. 
\end{cor}
\begin{proof}
From the previous lemma we deduce that the solution $\omega_t$ is given by the image of the horizontal lifting of $g_t$ on $\Cal{S}_1$ through $\Omega$. The second part follows from the fact that condition \eqref{flat} guarantees that the path $(g_t)_x$ lies in a maximal flat subspace of $S^2_+T^*_xM$ at any $x\in M$. 
\end{proof}
Let us observe that the map $\Omega$ is equivariant under the action of $B_g^h$. In other words if $s=(g,\psi)\in\Cal{S}_1$ and $h$ is a metric then
$$\Omega(\tilde{B}_h^gs)=B^{g}_h\Omega(s),$$
where the pointwise action of $B_g^h\in\mathrm{\GL}(TM)$ on $\Lambda^3T^*M$ is the standard one. In particular the solution of \eqref{ode} is given by $\omega_t=B_{g_t}^g\omega$, certainly if \eqref{flat} holds. \par
Generically the map $\omega\mapsto B_h^g\omega$  sends a positive three-form with metric $g$ to another with metric $h$ but it does not preserve the horizontal distribution.
\begin{prop}\label{derivative}
Let $\omega$ be a positive three-form with underlying metric $g$. Then the map
$$\Cal{R}\ni h\mapsto \Cal{B}^\omega(h)\in\Omega^3_+,\quad \Cal{B}^\omega(h)=B_h^g\omega,$$
sends a metric $h$ to a postive-three form compatible with $h$. Moreover it is Fréchet differentiable with tangent map
$$D\Cal{B}_h^\omega(k)=X(k)\neg \star_h \left(B_h^g\omega\right) +3i_{B_h^g\omega}(K(k)),\quad k\in T_h\Cal{R},$$
where
\begin{align*}
K(k)_{ab}&=\frac{1}{2}g_{(a|m}h^{mn}k_{np}g^{pq}h_{q|b)},\\
X(k)^{a}&=\frac{1}{2}g_{[i|m}h^{mn}k_{np}g^{pq}h_{q|j]}\omega^{ija}.
\end{align*}
\end{prop}
\begin{proof}
We compute the linearisation $D\Cal{B}^\omega_h(k)$. Write $\bar{\omega}=B^g_h\omega$ and $B_t=B^g_{h_t}$, for $h_t=h+tk$.
\begin{align*}
\frac{d}{dt}\left.B_t{\omega}\right|_{t=0}=&\left.\frac{d}{dt}\right|_{t=0}(B_tB_0^{-1})B_0\omega=\left(\frac{d}{dt}\left.{B}_t\right|_{t=0}B_0^{-1}\right)^*B_0\omega,\\
&=-\left(B_0\frac{d}{dt}\left.B^{-1}_t\right|_{t=0}\right)^*B_0\omega,
\end{align*}
where $(.)^*$ denotes the pointwise action of $\mathrm{End}(TM)$ on $\Lambda^3T^*M$. So
\begin{align*}
\frac{d}{dt}\left.B_t{\omega}\right|_{t=0}=3i_{B_0\omega}\left(B_0\frac{d}{dt}\left.B^{-1}_t\right|_{t=0}\right).
\end{align*}
The argument of $i_{B_0\omega}$ on the right side reads as $g_{ap}h^{pq}\left(1/2 k_{qr}g^{rb}\right)$, thus, setting $H_{ab}=1/2g_{ap}h^{pq}k_{qr}g^{rs}h_{sb},$ we derive
$$\frac{d}{dt}\left.B_t{\omega}\right|_{t=0}=3H^{\phantom{a}}_{[a|i|}\bar{\omega}^i_{\phantom{i}bc]}.$$
Finally, recalling the identity $\bar{\omega}_{ab}^{\phantom{ab}m}(\star_h\bar{\omega})_{mpqr}^{\phantom{m}}=3(h_{a[p}\bar{\omega}_{qr]b}-h_{b[p}\bar{\omega}_{qr]a})$, true for any fundamental three-form (see \cite{Gri}), we obtain the claimed formula for
\begin{align*}
K(k)_{ab}=H_{(ab)},\quad X(k)^a= H_{[bc]}\varphi^{bca}.
\end{align*} 
\end{proof}
By combining these two results we immediately obtain the following corollary.
\begin{cor}
Let $\omega$ be a positive-three form with underlying metric $g$, and let $h$ be a nowhere vanishing symmetric $(2,0)$-tensor. We denote by $\Cal{R}_g^h$ the space of Riemannian metrics $k$ such that the $g$-symmetric endomorphisms $k_{ai}g^{ib}$ and $h_{ai}g^{ib}$ commute. Then the map
$$\begin{CD} \Cal{B}^\omega:\;\Cal{R}_g^h@>>>\Omega^3_+,\end{CD}$$
embeds $\Cal{R}_g^h$ into a horizontal submanifold (in the sense of Fréchet).
\end{cor}
 



\begin{thebibliography}{CS79}

\bibitem{Am16}
Ammann, B. and Weiss, H. and Witt, F.
\textit{A spinorial energy functional: critical points and gradient flow.}
Mathematische Annalen, 365 (2016), no. 3, 1559-1602.
\bibitem{Ba17}
Bagaglini, L.
\textit{A flow of isometric $\Gtwo$-structures. Short-time existence.}
arXiv: (2017).

\bibitem{Chi}
Agricola, I. and Chiossi, G. Simon and Friedrich, T. and Höll, J.
\textit{Spinorial description of $\mathrm{SU}(3)$ and $\mathrm{G}_2$-structures.} 
Journal of Geometry and Physics. 98 (2015).

\bibitem{Vez}
Bedulli, L. and Vezzoni, L. 
\textit{A parabolic flow of balanced metrics}. 
J. reine angew. Math. 723 (2017), 79–99.

\bibitem{BG92}
Bourguignon, J. P. and Gauduchon, P.
\textit{Spineurs, opérateurs de Dirac et variations de
métriques}.
Comm. Math. Phys. 144 (1992), no. 3, 581-599.

\bibitem{Bry} 
Bryant, Robert L. 
\textit{Some Remarks on $\mathrm{G}_2-$Structures}. 
Proceedings of Gökova Geometry-Topology Conference (2005).

\bibitem{BX}
Bryant, Robert L. and Xu, F.
\textit{Laplacian Flow for Closed $\mathrm{G}_2$-structures: Short Time Behavior}.
\href{https://arxiv.org/abs/1101.2004}{https://arxiv.org/abs/1101.2004} (2011).

\bibitem{DeT83}
De Turck, Dennis M.
\textit{Deforming metrics in the direction of their Ricci tensors}.
J. Differential Geom. 18 (1983), no. 1, 157-162.

\bibitem{Fri97}
Friedrich, T. and Kath, T. and Moroianu, A. and Semmelmann, U.
\textit{On nearly parallel $\Gtwo$-structures}.
Journal of Geometry and Physics
Volume 23, Issues 3–4, November (1997), 259-286

 
\bibitem{Gri}
Grigorian, S.
\textit{Short-time behaviour of a modified Laplacian coflow of $\mathrm{G}_2-$structures}.
Advances in Mathematics. 248 (2013), 378-415.

\bibitem{Gri2}
Grigorian, S.
\textit{$\mathrm{G}_2$-structures and octonion bundles}.
Adv. Math. 308 (2017), 142-207.

\bibitem{Ham1}
Hamilton, R. S.
\textit{The inverse function theorem of Nash and Moser}.
Bull. Amer. Math. Soc. (N.S.) 7 (1982), no. 1, 65–222.

\bibitem{Ham2}
Hamilton, R. S.
\textit{Three-manifolds with positive Ricci curvature}.
J. Differential Geom. 17 (1982), no. 2, 255–306.

\bibitem{Hit03}
Hitchin, N.
\textit{Stable forms and special metrics}.
Global differential geometry: the mathematical legacy of Alfred Gray (Bilbao, 2000), number 288 in Contemp. Math (2001), 70-89.

\bibitem{SG}
Lawson, H. B. and Michelsohn, M. L.
\textit{Spin Geometry}.
Princeton University Press (1989).

\bibitem{WW12}
Weiss, H. and Witt, F.
\textit{A heat flow for special metrics}.
Adv. Math. 231 (2012), no. 6, 3288-3322.
\end{thebibliography}
\end{document}